\documentclass[reqno]{amsart}

\usepackage[letterpaper,margin=1.25in]{geometry}

\usepackage{verbatim} %

\usepackage{amsmath,amsthm,amssymb}
\usepackage{mathtools} %
\usepackage[all]{xy}
\SilentMatrices
\usepackage{url}

\newtheorem{de}{Definition}[section]
\newtheorem{lem}[de]{Lemma}
\newtheorem{prop}[de]{Proposition}
\newtheorem{cor}[de]{Corollary}
\newtheorem{thm}[de]{Theorem}
\newtheorem{thmintro}{Theorem}

\theoremstyle{remark}
\newtheorem{rem}[de]{Remark}
\newtheorem{ex}[de]{Example}

\newcommand{\dfn}[1]{\textbf{\boldmath{#1}}}

\DeclareMathOperator{\stab}{stab}

\newcommand{\ra}{\to}

\newcommand{\N}{\ensuremath{\mathbb{N}}}
\newcommand{\Z}{\ensuremath{\mathbb{Z}}}
\newcommand{\R}{\ensuremath{\mathbb{R}}}

\newcommand{\cC}{\mathcal{C}}
\newcommand{\cG}{\mathcal{G}}
\newcommand{\DS}{{\mathcal{DS}}}

\newcommand{\ddt}{{\textstyle\frac{d}{dt}}}
\newcommand{\id}{\mathrm{id}}
\newcommand{\pr}{\mathrm{pr}}

\newcommand{\io}{\iota}

\newcommand{\Diff}{\mathrm{Diff}}
\newcommand{\Vect}{\mathrm{Vect}}
\newcommand{\Set}{\mathrm{Set}}

\usepackage{ifpdf}
\ifpdf  \usepackage[pdftex,bookmarks=false]{hyperref}
\else   \usepackage[hypertex]{hyperref}
\fi
\usepackage{xcolor}
\definecolor{darkgreen}{rgb}{0,0.45,0} 
\hypersetup{colorlinks,urlcolor=blue,citecolor=darkgreen,linkcolor=blue,linktocpage=false}

\title{Tangent spaces of bundles and of\break filtered diffeological spaces}
\author{J. Daniel Christensen}
\email{jdc@uwo.ca}
\address{Department of Mathematics,
         University of Western Ontario,
         London, Ontario,
         Canada}

\author{Enxin Wu}
\email{enxin.wu@univie.ac.at}
\address{DIANA Group,
         Faculty of Mathematics,
         University of Vienna,
         Oskar-Morgenstern-Platz 1,
         1090 Vienna,
         Austria}

\date{June 30, 2016}

\begin{document}



\begin{abstract}
We show that a diffeological bundle gives rise to an exact sequence of internal tangent spaces.
We then introduce two new classes of diffeological spaces, which we call \dfn{weakly filtered}
and \dfn{filtered} diffeological spaces, whose tangent spaces are easier to understand.
These are the diffeological spaces whose categories of pointed plots are (weakly) filtered.
We extend the exact sequence one step further in the case of a diffeological bundle with
filtered total space and base space.
We also show that the tangent bundle $T^H X$ defined by Hector is a diffeological vector space
over $X$ when $X$ is filtered or when $X$ is a homogeneous space,
and therefore agrees with the dvs tangent bundle introduced by the authors in a previous paper.
\end{abstract}

\setcounter{tocdepth}{1}
\maketitle

\vspace*{-19pt}

\tableofcontents

\vspace*{-20pt}

\section{Introduction}

Diffeological spaces are elegant generalizations of smooth manifolds which include
many more examples, such as orbifolds, singular spaces and function spaces.
This makes them convenient for many applications, ranging from mathematical physics
to homological algebra~\cite{Wu} and homotopy theory~\cite{hDiff}.
In~\cite{CW}, we studied tangent spaces of diffeological spaces, focusing
on the internal tangent space $T_x(X)$.
Our results include many computations and general properties.
The union $T^H X$ of these tangent spaces was given the structure of a diffeological
space by Hector~\cite{He}, but it turns out that in general this diffeology does
not have the desired properties.  For example, the addition operation 
$T^H X \times_X T^H X \to T^H X$ is not smooth in general.
In~\cite{CW}, we enlarged the diffeology in a way that corrects this problem,
and denoted the result $T^{dvs} X$.
We gave examples where $T^H X$ is different from $T^{dvs} X$, and examples where
they agree, but the precise conditions under which they agree are not known.

In this paper, we have three main results.
For the first, we recall some terminology.
A \dfn{diffeological bundle with fiber $F$} is a smooth map $\pi : E \to B$ such
that for any smooth map $f : \R^n \to B$, the pullback of $\pi$ along $f$
is a trivial bundle with fiber $F$.

The first result is:

\begin{thmintro}[Theorem~\ref{th:bundle-exact2}]\label{th:1}
Suppose $\pi : E \to B$ is a diffeological bundle.
Let $e \in E$ and $b = \pi(e)$, and
write $\io : F \to E$ for the inclusion of the fiber over $b$.
Then
\[
\xymatrix{T_e(F) \ar[r]^{\io_*} & T_e(E) \ar[r]^-{\pi_*} & T_b(B) \ar[r] & 0}
\]
is an exact sequence of vector spaces.  Moreover, both
$\pi_*: T_e^H(E) \ra T_b^H(B)$ and $\pi_*: T_e^{dvs}(E) \ra T_b^{dvs}(B)$ are subductions.
\end{thmintro}

Diffeological bundles are pervasive, so this result applies to many situations.
An immediate corollary is that $T^H_{eH}(G/H)$ and $T^{dvs}_{eH}(G/H)$ agree
for any homogeneous space $G/H$ (Corollary~\ref{co:group-exact}).
In fact, we show that the bundles $T^H(G/H)$ and $T^{dvs}(G/H)$ agree (Proposition~\ref{prop:H=dvs-homogeneous}).

We can strengthen this theorem under an additional hypothesis.
The \dfn{germ category} $\cG(B,b)$ of a pointed diffeological space $(B,b)$
is the category whose objects are the smooth pointed maps $(U,0) \to (B,b)$,
for $U$ open in some $\R^n$, and whose morphisms are smooth germs
$(U,0) \to (V,0)$ commuting with the maps to $(B,b)$.
We say that a pointed diffeological space $(B,b)$ is \dfn{filtered} 
if the category $\cG(B,b)$ is filtered.
Manifolds are filtered, as is any irrational torus $T_{\theta}$ and any fine diffeological vector space.

\begin{thmintro}[Theorem~\ref{th:i-injective}]\label{th:2}
In the setting of Theorem~\ref{th:1}, if $(E,e)$ and $(B,b)$ are filtered, then $\io_*$ is injective.
\end{thmintro}

Note that an inclusion $A \subseteq X$ does not in general induce an injection
on tangent spaces.

Our final result also involves filtered diffeological spaces.
We say that $T^H X$ is a \dfn{diffeological vector space over $X$} if
scalar multiplication $T^H X \times \R \to T^H X$ and
addition $T^H X \times_X T^H X \to T^H X$ are smooth.
The diffeology on $T^{dvs} X$ is the smallest diffeology 
containing Hector's diffeology and making $T^{dvs} X$ into a
diffeological vector space over $X$.

\begin{thmintro}[Theorems~\ref{th:H=dvs} and~\ref{th:fine}]\label{th:3}
If $X$ is filtered, then $T^H X$ is a diffeological vector space over $X$.
Therefore, $T^H X$ and $T^{dvs} X$ agree.
Moreover, each $T_x(X)$ has the fine diffeology.
\end{thmintro}

Thus, in addition to homogeneous spaces, we have another class of diffeological spaces for which we know that these agree.

\medskip

The paper is organized as follows.
In Section~\ref{s:background}, we provide a brief summary of diffeological spaces,
diffeological bundles and internal tangent spaces.
In Section~\ref{s:exact}, we prove Theorem~\ref{th:1}, using a lemma which characterizes
the zero vectors in an internal tangent space.
In Section~\ref{s:filtered}, we define filtered diffeological spaces and weakly filtered
diffeological spaces.  We gives examples of both, and then prove Theorem~\ref{th:2}.
We give some other properties of weakly filtered diffeological spaces,
and conclude with Theorem~\ref{th:3}.

\subsection*{Conventions}
Throughout this paper, we make the following assumptions, unless otherwise stated.
Every vector space is over the field $\R$ of real numbers, and every linear function is $\R$-linear.
All manifolds are smooth, finite-dimensional, Hausdorff, second countable and 
without boundary, and are equipped with the standard diffeology.
Every subset of a diffeological space is equipped with the sub-diffeology,
every quotient is equipped with the quotient diffeology,
every product is equipped with the product diffeology,
and every function space is equipped with the functional diffeology.
All maps, sections and germs are assumed to be smooth.
We use the word ``function'' for a set map that may not be smooth.

\section{Background}\label{s:background}

In this section, we briefly recall some background on diffeological spaces,
including the notion of diffeological bundle and the internal tangent space.
For further details, we recommend the standard textbook~\cite{I}.
For a concise introduction to diffeological spaces, we recommend~\cite{CSW},
particularly Section~2 and the introduction to Section~3.
The material on tangent spaces is from~\cite{CW}.

\subsection{Basics}

\begin{de}[\cite{So2}]\label{de:diffeological-space}
A \dfn{diffeological space} is a set $X$
together with a specified set of functions $U \ra X$ (called \dfn{plots})
for each open set $U$ in $\R^n$ and each $n \in \N$,
such that for all open subsets $U \subseteq \R^n$ and $V \subseteq \R^m$:
\begin{enumerate}
\item (Covering) Every constant function $U \ra X$ is a plot;
\item (Smooth Compatibility) If $U \ra X$ is a plot and $V \ra U$ is smooth,
then the composite $V \ra U \ra X$ is also a plot;
\item (Sheaf Condition) If $U=\cup_i U_i$ is an open cover
and $U \ra X$ is a function such that each restriction $U_i \ra X$ is a plot,
then $U \ra X$ is a plot.
\end{enumerate}

A function $f:X \rightarrow Y$ between diffeological spaces is
\dfn{smooth} if for every plot $p:U \ra X$ of $X$,
the composite $f \circ p$ is a plot of $Y$.
\end{de}

Write $\Diff$ for the category of diffeological spaces and smooth maps.
Given two diffeological spaces $X$ and $Y$,
we write $C^\infty(X,Y)$ for the set of all smooth maps from $X$ to $Y$.
An isomorphism in $\Diff$ will be called a \dfn{diffeomorphism}.

Every manifold $M$ is canonically a diffeological space with the
plots taken to be all smooth maps $U \ra M$ in the usual sense.
We call this the \dfn{standard diffeology} on $M$.
It is easy to see that smooth maps in the usual sense between
manifolds coincide with smooth maps between them with the standard diffeology.

For a diffeological space $X$ with an equivalence relation~$\sim$,
the \dfn{quotient diffeology} on $X/{\sim}$ consists of all functions
$U \ra X/{\sim}$ that locally factor through the quotient map $X \to X/{\sim}$ via plots of $X$.
A \dfn{subduction} is a map diffeomorphic to a quotient map.
That is, it is a map $X \to Y$ such that the plots in $Y$
are the functions that locally lift to $X$ as plots in $X$.

For a diffeological space $Y$ and a subset $A$ of $Y$,
the \dfn{sub-diffeology} consists of all functions $U \ra A$ such that 
$U \ra A \hookrightarrow Y$ is a plot of $Y$.

The \dfn{discrete diffeology} on a set is the diffeology whose plots are
the locally constant functions.
\medskip

The category of diffeological spaces is very well-behaved:

\begin{thm}
The category $\Diff$ is complete, cocomplete and cartesian closed.
\end{thm}

The descriptions of limits, colimits and function spaces are quite
simple, and are concisely described in~\cite[Section~2]{CSW}.
The underlying set of a (co)limit is the (co)limit of the underlying sets.
The underlying set of the function space from $X$ to $Y$ is the set $C^\infty(X, Y)$.
\medskip

We can associate to every diffeological space the following topology:

\begin{de}[\cite{I1}]
Let $X$ be a diffeological space.
A subset $A$ of $X$ is \dfn{$D$-open} if $p^{-1}(A)$ is open in $U$ for
each plot $p : U \to X$.
The collection of $D$-open subsets of $X$ forms a topology on $X$ called the \dfn{$D$-topology}.
\end{de}

We will make use of the concepts of diffeological group and
diffeological vector space at several points, so we present them here.

\begin{de}\label{def:diff-group}
A \dfn{diffeological group} is a group object in $\Diff$.
That is, a diffeological group is both a diffeological space and a group
such that the group operations are smooth.
\end{de}

\begin{de}\label{def:dvs}
A \dfn{diffeological vector space} is a vector space object in $\Diff$.
More precisely, it is both a diffeological space and a vector space
such that the addition and scalar multiplication are smooth.
\end{de}

\subsection{Diffeological bundles}\label{ss:bundles}

Diffeological bundles are analogous to fiber bundles, but
are more general than the most obvious notion of locally trivial bundle.
We review diffeological bundles in this subsection,
and refer the reader to~\cite{I} for more details.

\begin{de}\label{de:diffbundle}
Let $F$ be a diffeological space.
A \dfn{diffeological bundle with fiber $F$} is a map $f: X \to Y$
between diffeological spaces such that for every plot $p: U \to Y$,
the pullback $p^*(f)$ of $f$ along $p$ is locally trivial with fiber $F$.
That is, there is an open cover $U = \cup_i U_i$ such that for each $i$
the restriction of $p^*(f)$ to $U_i$ is diffeomorphic, over $U_i$, to $U_i \times F$.
\end{de}

In~\cite{I}, diffeological bundles are defined using groupoids,
but~\cite[8.9]{I} shows that the definitions are equivalent.

\begin{ex}
Every smooth fiber bundle over a manifold is a diffeological bundle.
\end{ex}

\begin{prop}[{\cite[8.15]{I}}]\label{pr:diffgroup-diffbundle}
Let $G$ be a diffeological group, and let $H$ be a subgroup of $G$.
Then the projection $G \ra G/H$ is a diffeological bundle with fiber $H$,
where $G/H$ is the set of left (or right) cosets of $H$ in $G$ with the quotient diffeology.
\end{prop}

Note that we are \emph{not} requiring the subgroup $H$ to be closed.

\subsection{The internal tangent space}\label{ss:tangent}
Let $\DS_0$ be the category with
objects all connected open neighbourhoods of $0$ in $\R^n$ for all $n \in \N$
and morphisms the maps between them sending $0$ to $0$.
Given a pointed diffeological space $(X,x)$,
we define $\DS_0/(X,x)$ to be the category
with objects the pointed plots $p:(U,0) \ra (X,x)$ such that $U$ is in $\DS_0$,
and morphisms the commutative triangles
\[
\xymatrix@C5pt{U \ar[dr]_-p \ar[rr]^f & & V \ar[dl]^-q \\ & X , }
\]
where $p$ and $q$ are in $\DS_0/(X,x)$ and $f$ is a map with $f(0) = 0$.
We call $\DS_0/(X,x)$ the \dfn{category of plots of $X$ centered at $x$}.
It is the comma category of the natural functor from $\DS_0$ to
$\Diff_{*}$, the category of pointed diffeological spaces.

\begin{de}[\cite{He}]
Let $(X,x)$ be a pointed diffeological space.
The \dfn{internal tangent space} $T_x(X)$ of $X$ at $x$ is
the colimit of the composite of functors
$\DS_0/(X,x) \ra \DS_0 \ra \Vect$,
where $\Vect$ denotes the category of vector spaces and linear functions,
the first functor is the forgetful functor,
and the second functor is $T_0$, the usual tangent functor at $0$.
Given a pointed plot $p : (U,0) \to (X,x)$ and an element $u \in T_0(U)$,
we write $p_*(u)$ for the element these represent in the colimit.
This construction naturally defines a functor $T:\Diff_* \ra \Vect$.
\end{de}

\begin{rem}\label{rem:germcat}
The category of plots of a diffeological space centered at a point
is usually complicated, and it is usually more convenient to work
with the following variant.
This variant plays a central role in the present paper.
Let $\cG(X,x)$ be the category with the same objects as $\DS_0/(X,x)$,
but with morphisms the germs at $0$ of morphisms in $\DS_0/(X,x)$.
In more detail, morphisms from $p: (U,0) \ra (X,x)$ to $q: (V,0) \ra (X,x)$ in $\cG(X,x)$
consist of equivalence classes of maps $f : W \to V$,
where $W$ is an open neighborhood of $0$ in $U$ and $p|_W = q \circ f$.
Two such maps are equivalent if they agree on an open neighborhood of $0$ in $U$.
The internal tangent space can equivalently be defined using
$\cG(X,x)$ in place of $\DS_0/(X,x)$.
\end{rem}

It is shown in~\cite[Proposition 3.7]{CW} that
$T_{(x_1,x_2)}(X_1 \times X_2) \cong T_{x_1}(X_1) \times T_{x_2}(X_2)$
for any pointed diffeological spaces $(X_1, x_1)$ and $(X_2, x_2)$.

The \dfn{internal tangent bundle} $TX$ is the disjoint union of $T_x(X)$ for $x \in X$.
It can be equipped with two natural diffeologies, indicated
by superscripts: $T^H X$ and $T^{dvs} X$.
See~\cite[Subsection~4.1]{CW} for the descriptions of these diffeologies.
The fibers with the sub-diffeologies are written $T^H_x X$ and $T^{dvs}_x X$, respectively.

\section{Exact sequences of tangent spaces}\label{s:exact}

In this section, we show that a diffeological bundle gives rise to an
exact sequence of tangent spaces.

In order to prove our main result, we first need a lemma 
which characterizes the zero vectors in internal tangent spaces.

\begin{lem}\label{lem:kernel2}
Let $(X,x)$ be a pointed diffeological space,
let $p_i: (U_i,0) \ra (X,x)$ be distinct pointed plots, for $i = 1, \ldots, k$,
and let $u_i \in T_0(U_i)$ be such that $\sum_i (p_i)_*(u_i) = 0 \in T_x(X)$.
Then there exist pointed plots $p_i : (U_i, 0) \to (X, x)$ for $i = k+1, \ldots, n$,
distinct from each other and the earlier plots,
finitely many vectors $v_{ij} \in T_0(U_i)$ with 
\[
  \sum_j v_{ij} = \begin{cases}
                   u_i, & \text{if } i \leq k, \\
                   0  , & \text{if } i > k,
                 \end{cases}
\]
pointed plots $q_{ij} : (V_{ij}, 0) \to (X, x)$, and
germs $f_{ij} : (U_i, 0) \to (V_{ij}, 0)$ with $p_i = q_{ij} \circ f_{ij}$ as germs at $0$,
such that for each $q \in \{ q_{ij} \}$,
\[
  \sum_{\{(i,j) \mid q_{ij} = q\}} \, (f_{ij})_*(v_{ij}) = 0
\]
in $T_0(V_{ij})$.
\end{lem}

\begin{proof}
From the description of $T_x(X)$ as a colimit indexed by the
category $\cG(X,x)$, we can describe $T_x(X)$ as a quotient vector space $F/R$.
Here $F = \oplus_r T_0(U_r)$, where the sum is indexed over pointed plots $r : (U_r,0) \to (X,x)$,
and we write an element of the summand indexed by $r$ as $(r, v)$.
The linear subspace $R$ is the span of the ``basic relations'' $(r,v) - (q, g_*(v))$,
where $r : (U_r, 0) \to (X,x)$ and $q : (U_q,0) \to (X,x)$ are pointed plots,
$g : (U_r,0) \to (U_q,0)$ is a germ with $r = q \circ g$ as germs at $0$,
and $v$ is in $T_0(U_r)$.

Since $\sum_i (p_i)_*(u_i)$ is zero in $T_x(X)$, $\sum_i (p_i,u_i) \in F$ can be expressed as
\[
  \sum_{i=1}^k (p_i, u_i) = \sum_{\ell=1}^m \, (r_\ell, v_\ell) - (q_\ell, (g_\ell)_*(v_\ell))
\]
for $r_\ell, q_\ell, g_\ell$ and $v_\ell$ as in the previous paragraph.
Let $d$ be the maximum of the dimensions of the domains $U_{r_\ell}$ and $U_{q_\ell}$ of 
the plots appearing on the right-hand side of this equation.
For each $\ell$, choose an extension of $q_\ell$ to a plot $s_\ell$ where $U_{s_\ell}$ has
dimension bigger than $d$, so $q_\ell$ is the composite
\[
  \xymatrix{(U_{q_\ell},0) \ar[r]^{h_\ell} & (U_{s_\ell}, 0) \ar[r]^{s_\ell} & (X,x).}
\]
For example, take $U_{s_{\ell}} = U_{q_{\ell}} \times \R^{d+1}$ with $h_{\ell} = i_1$
and $s_{\ell} = q_{\ell} \circ \pr_1$, where $i_1$ and $\pr_1$ are the inclusion
into and projection onto the first factor.
Then
\[
\begin{aligned}
  \sum_{i=1}^k (p_i, u_i) &= \sum_{\ell=1}^m \, (r_\ell, v_\ell) - (q_\ell, (g_\ell)_*(v_\ell)) \\
 &= \sum_{\ell=1}^m\, \big[(r_\ell, v_\ell) - (s_\ell, (h_\ell)_*(g_\ell)_* (v_\ell))\big]
                  - \big[(q_\ell, (g_\ell)_*(v_\ell)) - (s_\ell, (h_\ell)_*(g_\ell)_* (v_\ell))\big]. \\
\end{aligned}
\]
This is a sum of basic relations, where all of the right-hand terms involve plots $s_\ell$
that are distinct from all of the plots $r_\ell$ and $q_\ell$ appearing in the left-hand terms
of each basic relation.
Relabel the $r_\ell$'s and $q_\ell$'s to be $p_1, \ldots, p_i, \ldots, p_n$, all distinct, and agreeing
with the given plots when $i \leq k$.
Since we may have had collisions in the relabelling, we label the
vectors $v_\ell$ and $(g_\ell)_*(v_\ell)$ as $v_{ij}$ when they are associated to the plot $p_i$.
Similarly, the germs $h_\ell \circ g_\ell$ and $h_\ell$ that appear in the above relations
are relabelled as $f_{ij}$, and $s_\ell$ is relabelled as $q_{ij}$.

In this new notation,
\[
  \sum_{i=1}^k (p_i, u_i) = \sum_{i=1}^n \sum_j \, (p_i, v_{ij}) - (q_{ij}, (f_{ij})_*(v_{ij})).
\]
Since the $q_{ij}$'s are distinct from all of the $p_i$'s, it must be the
case that the terms $(q_{ij}, (f_{ij})_*(v_{ij}))$ sum to zero, and the last
displayed equation of the statement expresses this fact grouped by the $q_{ij}$'s
that happen to agree.
Similarly, using that the $p_i$'s are distinct, and comparing terms on both
sides of the equation, gives the other displayed equation in the statement.
\end{proof}

\begin{rem}
It is not hard to show that the previous lemma in fact characterizes the zero tangent vectors.
The characterization is complicated because the category $\cG(X,x)$ is not filtered in general.
Compare Remark~\ref{rem:colimit}.
\end{rem}

\begin{thm}\label{th:bundle-exact2}
Suppose $\pi : E \to B$ is a diffeological bundle.
Let $e \in E$ and $b = \pi(e)$, and
write $\io : F \to E$ for the inclusion of the fiber over $b$.
Then
\[
\xymatrix{T_e(F) \ar[r]^{\io_*} & T_e(E) \ar[r]^-{\pi_*} & T_b(B) \ar[r] & 0}
\]
is an exact sequence of vector spaces.  Moreover, both
$\pi_*: T_e^H(E) \ra T_b^H(B)$ and $\pi_*: T_e^{dvs}(E) \ra T_b^{dvs}(B)$ are subductions.
\end{thm}

\begin{proof}
Since $\pi: E \ra B$ is a diffeological bundle, it is easy to see that $\pi_*$ is surjective,
because plots $(U,0) \to (B,b)$ lift locally to $(E,e)$.

Since $\pi \circ \io$ is constant, $\pi_* \circ \io_* = 0$.

To complete the proof of exactness, we need to show that if $v \in T_e(E)$ and $\pi_*(v) = 0$,
then $v = \io_*(w)$ for some $w \in T_e(F)$.
To illustrate the argument, we first treat the case that $v = s_*(\ddt)$
for some pointed plot $s: (\R,0) \ra (E,e)$,
and assume that there exist a pointed plot $q: (V,0) \ra (B,b)$ and a
germ $f: (\R,0) \ra (V,0)$ such that $f_*(\ddt)=0$ and $\pi \circ s = q \circ f$ as germs at $0$. 
We may assume that $V$ is an open disk in some Euclidean space.
Since $\pi$ is a diffeological bundle, we have a diffeomorphism
$\phi : q^* E \to V \times F$ commuting with the projections to $V$.
By replacing $\phi$ with the map $(\id_V \times \psi) \circ \phi$,
where $\psi$ is a diffeomorphism $F \to F$, we can ensure that the
following diagram commutes
\begin{equation}\label{eq:phi}
  \vcenter{\xymatrix{
    & & F \ar[d]^\io \ar@/_9pt/[lld]_{i_2} \\
    V \times F \ar[r]^(0.6){\phi^{-1}} & q^* E \ar[r]^g & E , \\
  }}
\end{equation}
where $i_2(x) = (0, x)$ and $g$ is the natural map.
Define a germ $r : (\R, 0) \to F$ as the composite
\begin{equation}\label{eq:r}
  \xymatrix{
  \R \ar[r]^-{(f,s)} & q^* E \ar[r]^-{\phi} & V \times F \ar[r]^-{\pr_2} & F ,
  }
\end{equation}
where the first map is understood as a germ at $0 \in \R$. Then $r(0)=e$, using~\eqref{eq:phi}.
 From the diagram
\[
  \xymatrix{
  \R \ar[r]^-{(f,s)} \ar[drr]_f & q^* E \ar[r]^-{\phi} \ar[dr] & V \times F \ar[r]^-{\pr_2} \ar[d]^{\pr_1} & F \\
                             &                              & V & ,
  }
\]
it follows that the image of $\ddt \in T_0(\R)$ in $T_{(0,e)}(V \times F) \cong T_0(V) \oplus T_e(F)$
is concentrated in the right-hand factor, or, in other words, is fixed
by $(i_2 \circ \pr_2)_*$.
Therefore,
\[
    \io_* r_*(\ddt)
  = g_* (\phi^{-1})_* (i_2)_* (\pr_2)_* \phi_* (f,s)_* (\ddt)
  = g_* (\phi^{-1})_* \phi_* (f,s)_* (\ddt)
  = g_* (f,s)_* (\ddt)
  = s_* (\ddt),
\]
and so $s_*(\ddt)$ is in the image of $\io_*$.
The first equality uses diagrams~\eqref{eq:phi} and~\eqref{eq:r}.

To handle the general case, let $v$ be of the form $\sum_{i=1}^k (s_i)_*(u_i)$ for pointed plots $s_i : (U_i, 0) \to (E, e)$
and $u_i \in T_0(U_i)$.
Let $p_i = \pi \circ s_i$ for each $i$.
Then, since $\pi_*(v) = \sum_i (p_i)_*(u_i) = 0$, we can apply Lemma~\ref{lem:kernel2} to
obtain $p_i$, $i = k+1, \ldots, n$, $v_{ij}$, $q_{ij}$ and $f_{ij}$ as described there,
with $(X,x) = (B,b)$.
Since $\pi$ is a diffeological bundle, we can assume that each $p_i$ for $i > k$ is also
of the form $\pi \circ s_i$.
For each $i$ and $j$, choose a diffeomorphism $\phi_{ij} : q_{ij}^*E \to V_{ij} \times F$,
being sure to choose $\phi_{ij} = \phi_{i'j'}$ when $q_{ij} = q_{i'j'}$ and ensuring that
$g_{ij} \circ \phi_{ij}^{-1} \circ i_2 = \io$, where $g_{ij} : q_{ij}^* E \to E$ is the
natural map (cf.\ \eqref{eq:phi}).
Define $r_{ij} : (U_i, 0) \to (F,e)$ to be $\pr_2 \circ \phi_{ij} \circ (f_{ij}, s_i)$.
Observe that for $q \in \{ q_{ij} \}$
\[
\begin{aligned}
  (\pr_1)_* \bigg( \sum_{\{ i, j \,\mid\, q_{ij} = q \}} (\phi_{ij})_* (f_{ij}, s_i)_* (v_{ij})\bigg)
&= \sum_{\{ i, j \,\mid\, q_{ij} = q \}} (\pr_1)_* (\phi_{ij})_* (f_{ij}, s_i)_* (v_{ij}) \\
&= \sum_{\{ i, j \,\mid\, q_{ij} = q \}} (f_{ij})_* (v_{ij}) \\
&= 0 .
\end{aligned}
\]
Therefore, writing $g_q$ for $g_{ij}$ when $q = q_{ij}$, and similarly for $\phi_q$,
\[
\begin{aligned}
  \io_*\bigg(\sum_{i,j} (r_{ij})_*(v_{ij})\bigg)
&= \, \sum_{i,j} \ (g_{ij})_* (\phi_{ij}^{-1})_* (i_2)_* (\pr_2)_* (\phi_{ij})_* (f_{ij}, s_i)_* (v_{ij}) \\
&= \, \sum_{\mathclap{q \in \{q_{ij}\}}} \ (g_q)_* (\phi_q^{-1})_* (i_2)_* (\pr_2)_* \bigg(\sum_{\{i,j \,\mid\, q_{ij} = q\}} \!\! (\phi_{ij})_* (f_{ij}, s_i)_* (v_{ij}) \bigg)\\
&= \, \sum_{\mathclap{q \in \{q_{ij}\}}} \ (g_q)_* (\phi_q^{-1})_* \bigg(\sum_{\{i,j \,\mid\, q_{ij} = q\}} \!\! (\phi_{ij})_* (f_{ij}, s_i)_* (v_{ij}) \bigg)\\
&= \, \sum_{i,j} \ (g_{ij})_* (f_{ij}, s_i)_* (v_{ij})\\
&= \, \sum_{i=1}^n \ (s_i)_* \bigg( \sum_j v_{ij} \bigg)\\
&= \, \sum_{i=1}^k \ (s_i)_* (u_i) = v ,\\
\end{aligned}
\]
as required.

We next prove that $\pi_*: T_e^H(E) \ra T_b^H(B)$ is a subduction.
Let $p:U \ra T_b^H(B)$ be a plot.
We must show that it locally lifts to a plot of $T_e^H(E)$.
By definition of Hector's diffeology, $p$ is locally either constant
or of the form 
\[
\xymatrix{U \ar[r]^-{(\alpha,\beta)} & V \times \R^n = TV \ar[r]^-{Tq} & T^H B ,}
\]
where $n = \dim(V)$, $(\alpha,\beta)$ is smooth, and $q:V \ra B$ is a plot with $V$ a small ball.
We only need to deal with the second case.
Since $p$ lands in $T^H_b(B)$, $q \circ \alpha$ is the constant map at $b$.
Since $\pi$ is a diffeological bundle, we have the following pullback square in 
$\Diff$, where $h = g \circ \phi^{-1}$ in our earlier notation:
\[
\xymatrix{V \times F \ar[r]^-h \ar[d]_{\pr_1} & E \ar[d]^\pi \\ V \ar[r]_q & B.}
\]
So there exists a map $\gamma:U \ra F$ such that $h \circ (\alpha,\gamma)$ is the constant map 
$e$, and we have the following commutative diagram in $\Diff$:
\[
\xymatrix@C+10pt{U \times F \ar[r]^-{(\alpha,\beta) \times \eta} & TV \times T^H F \ar[r]^-{Th} \ar[d]_{\pr_1} & T^H E \ar[d]^{T \pi} \\
U \ar[r]_{(\alpha,\beta)} \ar[u]^{(1_U,\gamma)} & TV \ar[r]_{Tq} & T^H B,}
\]
where $\eta:F \ra T^H F$ is zero section. Note that the image of the composition 
$Th \circ ((\alpha,\beta) \times \eta) \circ (1_U,\gamma)$ is in $T_e^H(E)$,
and so we have produced the required lift.

The corresponding result for the $dvs$ diffeology then follows from~\cite[Remark~4.7]{CW}.
\end{proof}

As a special case, we have:

\begin{cor}\label{co:group-exact}
Let $G$ be a diffeological group, let $H$ be a subgroup of $G$, and let $e$ be the identity element. 
Then the natural maps $\io: H \hookrightarrow G$ and $\pi:G \twoheadrightarrow G/H$ induce the following 
exact sequence of vector spaces:
\[
\xymatrix{T_e(H) \ar[r]^{\io_*} & T_e(G) \ar[r]^-{\pi_*} & T_{eH}(G/H) \ar[r] & 0.}
\]
Moreover, $T_{eH}^H(G/H) = T_{eH}^{dvs}(G/H)$ and $\pi_*: T_e(G) \ra T_{eH}(G/H)$ is a subduction.
\end{cor}

\begin{proof}
As $\pi:G \ra G/H$ is a diffeological bundle with fiber $H$,
most of this follows directly from Theorem~\ref{th:bundle-exact2}.
We just need to prove that $T_{eH}^H(G/H) = T_{eH}^{dvs}(G/H)$.
First note that $T_e^H(G) = T_e^{dvs}(G)$, by~\cite[Theorem~4.15]{CW}.
By Theorem~\ref{th:bundle-exact2}, both $\pi_*: T_e^H(G) \ra T_{eH}^H(G/H)$ and 
$\pi_*: T_e^{dvs}(G) \ra T_{eH}^{dvs}(G/H)$ are subductions.
Since both maps have the same underlying function, 
it follows that $T_{eH}^H(G/H) = T_{eH}^{dvs}(G/H)$.
\end{proof}

In Proposition~\ref{prop:H=dvs-homogeneous}, we will generalize the last statement of the previous corollary.

\begin{ex}\label{ex:irrational-torus}
For $\theta$ an irrational number, the irrational torus $T_{\theta}$ is diffeomorphic
to $\R/(\Z+\theta \Z)$.
It follows immediately from Corollary~\ref{co:group-exact} that $T_{[0]}(T_\alpha)=\R$.
See~\cite[Example~3.23]{CW} for a direct computation.
\end{ex}

We don't know whether the function $\io_*$ is injective in general in
either Theorem~\ref{th:bundle-exact2} or Corollary~\ref{co:group-exact}.
We show in Theorem~\ref{th:i-injective} that it is injective under stronger hypotheses.

\section{Filtered diffeological spaces}\label{s:filtered}

In this section, we introduce filtered and weakly filtered diffeological spaces.
For example, manifolds are filtered, and diffeological groups are weakly filtered.
Our main result is that the inclusion map $\io : F \to E$ of a diffeological
bundle with filtered total space and base space induces an injection on tangent spaces, 
extending the exact sequence from the previous section.
We also show that when $X$ is filtered, Hector's internal tangent bundle $T^H X$
is a diffeological vector space over $X$, $T^H X$ agrees with $T^{dvs} X$,
and each $T_x(X)$ has the fine diffeology (Definition~\ref{de:fine}).

We say that a category $\cC$ is \dfn{weakly filtered} if it is non-empty, 
and for any two objects $C$ and $D$ in $\cC$ there exists an object
$E$ and morphisms $C \to E$ and $D \to E$.
$\cC$ is \dfn{filtered} if it is weakly filtered and
for any two parallel morphisms $f, g : C \to D$ in $\cC$
there exists a morphism $h : D \to E$ such that $h \circ f = h \circ g$.
Equivalently, $\cC$ is filtered if every finite diagram in $\cC$ has a cocone.

\begin{de}\label{de:filtered}
A pointed diffeological space $(X,x)$ is \dfn{(weakly) filtered} if the germ category $\cG(X,x)$ 
introduced in Remark~\ref{rem:germcat} is (weakly) filtered.
A diffeological space $X$ is \dfn{(weakly) filtered} if $(X,x)$ is (weakly) filtered for every 
$x \in X$.
\end{de}

\begin{lem}\label{le:filt} \ 
\begin{enumerate}
\item If $U$ is open in $\R^n$ and $u$ is in $U$, then $(U,u)$ is filtered.
\item If $(X_1, x_1), \ldots, (X_n, x_n)$ is a finite family of pointed diffeological spaces,
then $(\prod_{i=1}^n X_i,\, (x_i))$ is (weakly) filtered if and only if each $(X_i, x_i)$ is.
\item Being (weakly) filtered is local: if $(X,x)$ is a pointed diffeological space, 
and $A$ is a $D$-open neighborhood of $x$ in 
$X$, then $(X,x)$ is (weakly) filtered if and only if $(A,x)$ is.
\item If $(X,x)$ is a (weakly) filtered pointed diffeological space, and $(Y,y)$ is a retract of $(X,x)$ in 
$\Diff_*$, then $(Y,y)$ is also (weakly) filtered.
\end{enumerate}
It follows from (1) and (3) that manifolds are filtered.
\end{lem}

\begin{proof}
(1) follows from the fact that the germ category $\cG(U,u)$ has a terminal object, and
the rest are straightforward.
\end{proof}

The next proposition allows us to determine another class of filtered diffeological spaces.
We say that a pointed map $\pi : (E,e) \to (B,b)$ has the 
\dfn{unique germ lifting property} if each germ $g : (U,0) \to (B,b)$
lifts to a unique germ $f : (U, 0) \to (E,e)$ with $g = \pi \circ f$ as germs at $0$.
This is the case when $\pi$ is a diffeological bundle with discrete fiber.

\begin{prop}
If $(E,e)$ is filtered and $\pi : (E,e) \to (B,b)$ has the
unique germ lifting property, then $(B,b)$ is filtered.
\end{prop}

\begin{proof}
Suppose $D : I \to \cG(B,b)$ is a finite diagram.
For each object $i$ of $I$, we have a pointed plot $D_i : (U_i, 0) \to (B,b)$.
By assumption, this lifts to a pointed plot $D'_i : (U'_i, 0) \to (E,e)$
after restriction to an open neighbourhood $U'_i$ of $0$ in $U_i$.
For each morphism $\alpha : i \to j$ of $I$, $D(\alpha) : (U_i, 0) \to (U_j, 0)$
is a germ, and thus defines a germ $D'(\alpha) : (U'_i, 0) \to (U'_j, 0)$.
By uniqueness of germ lifting, $D'(\alpha)$ is a morphism in $\cG(E,e)$,
and so $D'$ is a functor $I \to \cG(E,e)$.
Since $(E,e)$ is filtered, this functor has a cocone $c_i' : D'(i) \to C$.
Then $c_i : D(i) \to \pi \circ C$ is a cocone for $D$, showing that $\cG(B,b)$ is filtered.
\end{proof}

\begin{cor}
If $G$ is a Lie group and $H$ is a subgroup such that the sub-diffeology is discrete,
then $G/H$ is filtered.
\end{cor}

\begin{ex}
For $\theta$ an irrational number, the irrational torus $T_{\theta}$ is diffeomorphic
to $\R/(\Z+\theta \Z)$ and is therefore filtered.
The sub-diffeology on $\Z+\theta \Z$ is discrete, but the sub-topology is not.
\end{ex}

We show in Corollary~\ref{co:fine-filtered} that fine diffeological vector spaces
are also filtered.

\begin{rem}
If $(E,e)$ is weakly filtered and $\pi : (E,e) \to (B,b)$ is a map such
that germs $(U,0) \to (B,b)$ lift to $(E,e)$ (not necessarily uniquely),
then a similar argument shows that $(B,b)$ is weakly filtered.
This applies, for example, when $\pi$ is a subduction such that
$\pi^{-1}(b) = \{ e \}$.
\end{rem}

\begin{ex}
A \dfn{diffeological orbifold} is a diffeological space which is locally
diffeomorphic to a quotient of a Euclidean space by a linear action of a
finite group.  
By the previous remark, orbifolds are weakly filtered.
However, orbifolds are not filtered in general.  For example, $\R/O(1)$ is not filtered.
Indeed, the quotient map $q : (\R, 0) \to (\R/O(1), [0])$ is a plot which
has two endomorphisms, $\pm 1 : (\R, 0) \to (\R, 0)$.
It is not hard to see that the plot $q$ does not factor as a germ at $0$
through another plot $r$ in a way that coequalizes the two endomorphisms.
\end{ex}

The remark below will be used in the proof of the following theorem.

\begin{rem}\label{rem:colimit}
Recall that the internal tangent space $T_x(X)$ is the colimit of the functor 
$\cG(X,x) \ra \Vect$ sending a pointed plot $(U,0) \to (X,x)$ to $T_0(U)$.
Let $T^{set}_x(X)$ be the result of taking the colimit in the category of sets
instead of vector spaces.
It is a standard fact that the forgetful functor from $\Vect$ to $\Set$ preserves filtered colimits,
so when $(X,x)$ is filtered, the natural function $T^{set}_x(X) \to T_x(X)$ is actually a bijection.
In particular, this implies that 
if $\sum_{i=1}^k (p_i)_*(u_i) = 0 \in T_x(X)$, where $p_i: (U_i,0) \ra (X,x)$ are pointed plots and $u_i \in T_0(U_i)$, then there 
exist a pointed plot $q:(V,0) \ra (X,x)$ and germs $f_i: (U_i,0) \ra (V,0)$ such that 
$p_i = q \circ f_i$ as germs at $0$ and $\sum_i (f_i)_*(u_i) = 0 \in T_0(V)$.
Compare this to the more complicated characterization of zero vectors in Lemma~\ref{lem:kernel2}.
\end{rem}

For diffeological bundles with filtered total space and base space, we can extend the exact
sequence of Theorem~\ref{th:bundle-exact2}.

\begin{thm}\label{th:i-injective}
Suppose $\pi : E \to B$ is a diffeological bundle.
Let $e \in E$ and $b = \pi(e)$, and
write $\io : F \to E$ for the inclusion of the fiber over $b$.
If $(E,e)$ and $(B,b)$ are filtered, then
\[
\xymatrix{0 \ar[r] & T_e(F) \ar[r]^{\io_*} & T_e(E) \ar[r]^-{\pi_*} & T_b(B) \ar[r] & 0}
\]
is an exact sequence of vector spaces.  Moreover, both
$\pi_*: T_e^H(E) \ra T_b^H(B)$ and $\pi_*: T_e^{dvs}(E) \ra T_b^{dvs}(B)$ are subductions.
\end{thm}

Note that if $A$ is a subset of $X$ with the sub-diffeology, the
inclusion does not in general induce an injection on tangent
spaces~\cite[Example~3.20]{CW}.

\begin{proof}
\newcommand{\tp}{p}
We only need to show that $\io_*$ is injective.
Let $v \in T_e(F)$ be represented as $v = \sum_{i=1}^k (\tp_i)_*(u_i)$ for
pointed plots $\tp_i : (U_i, 0) \to (F, e)$.
Suppose $\io_*(v) = 0$.
Then, since $(E,e)$ is filtered, there are commuting diagrams
\[
  \xymatrix{
    (U_i,0) \ar[r]^{\tp_i} \ar[d]_{f_i} & (F,e) \ar[d]^\io \\
    (V,0) \ar[r]_{q}           & (E,e), 
  }
\]
where $q$ is a pointed plot, each $f_i$ is a germ, and $\sum_i (f_i)_*(u_i) = 0$.
The germ $f_{i}$ can be thought of as a morphism between the pointed plots $\pi \circ \io \circ \tp_i$ and $\pi \circ q$ in $\cG(B,b)$.
The former is constant, and so the $0$ germ is also a morphism between these pointed plots.
Therefore, since $(B,b)$ is filtered, we can extend the diagram to
\[
  \xymatrix{
    (U_i,0) \ar[r]^{\tp_i} \ar[d]_{f_{i}} & (F,e) \ar[d]^\io \\
    (V,0) \ar[r]_{q} \ar[d]_{g}  & (E,e) \ar[d]^{\pi} \\
    (W,0) \ar[r]_r           & (B,b) ,
  }
\]
where $g$ is a germ with $g \circ f_{i}$ constant for each $i$, and $r$ is a pointed plot.
Without loss of generality, assume $W$ is a ball.
Then, since $\pi$ is a diffeological bundle, taking the pullback of $\pi$ along $r$
gives a trivial bundle, and we get
\[
  \xymatrix{
    U_i \ar[r]^{\tp_i} \ar[d]_{f_{i}}  & F \ar@{=}[r] \ar[d]_-{i_2} & F \ar[d]^\io \\
    V \ar[r]^-{q'}  \ar[dr]_{g} & W \times F \ar[d]_(0.4){\pr_1} \ar[r]^-{\alpha} & E \ar[d]^{\pi} \\
                          & W \ar[r]^r           & B ,
  }
\]
where $i_2(x) = (0,x)$ for $x \in F$,
and we have suppressed basepoints.
(As in the proof of Theorem~\ref{th:bundle-exact2}, we may need to adjust
the map $\alpha$ to make the upper-right square commute.
We may also need to shrink $V$ and the $U_i$, since $f_i$ is a germ.)
The map $q'$ is induced by the maps $g : V \to W$ and $q : V \to E$,
and the top-left square commutes by the universal property of the pullback.
Let $s : V \to F$ be the composite $\pr_2 \circ q'$.
Since $g \circ f_{i}$ is constant, $q' \circ f_{i}$ lands in $\{ 0 \} \times F$.
Thus $i_2 \circ s \circ f_{i} = i_2 \circ \pr_2 \circ q' \circ f_{i} = q' \circ f_{i} = i_2 \circ \tp_i$.
Since $i_2$ is injective, $s \circ f_{i} = \tp_i$.
Thus the germs $f_i$ are morphisms in $\cG(F,e)$ and the relation
$\sum_i (f_i)_*(u_i) = 0$ shows that the original vector $v$ is zero in $T_e(F)$.
\end{proof}

We don't know whether $\io_*$ is an induction in general.

\medskip

Before getting to our result about tangent bundles, we make some observations
about weakly filtered diffeological spaces.

\begin{prop}\label{prop:dgrp}
Every diffeological group is weakly filtered.
\end{prop}

\begin{proof}
By using the left or right multiplication map, it is clear that $(G,e)$ is isomorphic to $(G,g)$ in 
$\Diff_*$ for any $g \in G$, where $e$ denotes the identity element of $G$.
So it is enough to show that $(G,e)$ is weakly filtered.
For any pointed plots $p:(U,0) \ra (G,e)$ and $q:(V,0) \ra (G,e)$, define 
$r:U \times V \ra G$ by $r(u,v)=p(u)q(v)$, $f:U \ra U \times V$ by $f(u)=(u,0)$, 
and $g:V \ra U \times V$ by $g(v)=(0,v)$.
Then these functions are all smooth and they satisfy $r(0,0) = e$, $p = r \circ f$, 
and $q = r \circ g$.
\end{proof}

We will see in Example~\ref{ex:prodR} that diffeological groups are not in general filtered.

It follows from Proposition~\ref{prop:dgrp} that any homogeneous space $G/H$ is weakly filtered,
as is the function space $C^\infty(X,N)$, where 
$X$ is any diffeological space with compact $D$-topology, and $N$ is a manifold.

\medskip

Recall that for an arbitrary pointed diffeological space $(X,x)$, 
not every internal tangent vector $v \in T_x(X)$ can be written as $p_*(u)$ 
for some pointed plot $p:(U,0) \ra (X,x)$ and $u \in T_0(U)$.
When this is possible for every $v$, we say that $T_x(X)$ is \dfn{$1$-representable}.
See~\cite[Remark~3.8]{CW} for more background.

Here is a key property of weakly filtered pointed diffeological spaces:

\begin{prop}\label{pr:1-rep}
Let $(X,x)$ be weakly filtered. Then the internal tangent space $T_x(X)$ 
is $1$-representable.
\end{prop}

\begin{proof}
Let $\sum_{i=1}^n (p_i)_*(u_i) \in T_x(X)$ be any internal tangent vector, 
where $p_i:(U_i,0) \ra (X,x)$ is a pointed plot and $u_i \in T_0(U_i)$ for each $i$. 
By induction on $n$ using the definition of weakly filtered, there exists a pointed plot 
$q:(U,0) \ra (X,x)$ together with germs $f_i:(U_i,0) \ra (U,0)$ such that 
$p_i=q \circ f_i$ as germs at $0$. 
Then $\sum_{i=1}^n (p_i)_*(u_i)=q_*(\sum_{i=1}^n (f_i)_*(u_i))$, 
and hence $T_x(X)$ is $1$-representable.
\end{proof}

It follows that not every diffeological space is weakly filtered.
For example, if $X$ is the union of the axes in $\R^2$ with the sub-diffeology,
and $x$ is the origin, then $(X,x)$ is not weakly filtered,
since $T_x(X)$ is not $1$-representable (see~\cite[Example~3.19]{CW}).

\medskip

When $X$ is weakly filtered, each $T_x(X)$ is $1$-representable, and so
by~\cite[Remark~4.4]{CW}, scalar 
multiplication $\R \times T^H X \ra T^H X$ is smooth. 
When $X$ is filtered, addition is also smooth:

\begin{thm}\label{th:H=dvs}
If $X$ is filtered, then $T^H X$ is a diffeological vector space over $X$.
Therefore, $T^H X$ and $T^{dvs} X$ agree.
\end{thm}

\begin{proof}
Since scalar multiplication is smooth when $X$ is weakly filtered, it suffices to show that
addition $T^H X \times_X T^H X \to T^H X$ is smooth.
That is, if $p_1$ and $p_2$ are plots $U \to T^H X$ such that $\pi \circ p_1 = \pi \circ p_2$,
then $p_1 + p_2 : U \to T^H X$ is smooth. 
As smoothness is a local property, it is enough to show that for each $u$ in $U$,
$p_1 + p_2$ is smooth in a neighbourhood of $u$ in $U$.
By translating $U$, one can assume that $u = 0$, so from now on, we work locally around $0 \in U$
and choose $x = \pi(p_i(0)) \in X$ as our basepoint.
By definition, each $p_i$ is locally of the form $T q_i \circ f_i$, where
$f_i$ has codomain $T V_i$ and $q_i : V_i \to X$ is a plot.
(In general, one needs constant plots as well, but since each $T_x(X)$ is $1$-representable,
plots of the above form suffice.)
Since we are working locally, we can assume that each $p_i$ equals $T q_i \circ f_i$,
giving the following diagram:
\[
  \xymatrix{
                   & TV_i \ar[r]^-{T q_i} \ar[d]^{\pi_i} & T^H X \ar[d]^{\pi} \\
    U \ar[ur]^{f_i} & V_i \ar[r]^{q_i}            & X .
  }
\]
By translating $V_1$ and $V_2$ if necessary, we can assume that $\pi_i(f_i(0)) = 0$ for each $i$
and therefore that $q_1$ and $q_2$ are pointed plots.
Because $X$ is weakly filtered, by shrinking $U$, $V_1$ and $V_2$ if necessary, 
we see that $q_1$ and $q_2$ factor through a further pointed plot $q$,
so we can assume that $V_1 = V_2 = V$ and $q_1 = q_2 = q$.  So we have
\[
  \xymatrix{
                   & TV \ar[r]^-{T q} \ar[d]^{\pi_V} & T^H X \ar[d]^{\pi} \\
    U \ar@/^6pt/[ur]^{f_1} \ar@/_0pt/[ur]_(0.6){\!\!f_2}  \ar@/^0pt/[r]^{g_1} \ar@/_6pt/[r]_{g_2} 
                   & V \ar[r]^{q}            & X ,
  }
\]
where $g_i := \pi_V \circ f_i$.
From $\pi \circ p_1 = \pi \circ p_2$, we deduce that $q \circ g_1 = q \circ g_2$.
Since $X$ is filtered, by shrinking $U$ and $V$ if necessary, 
there exist a plot $r : W \to X$ and a map $h : V \to W$
such that $h \circ g_1 = h \circ g_2$ and $r \circ h = q$:
\[
  \xymatrix{
                   & TV \ar[r]^-{T h} \ar[d]^{\pi_V} & TW \ar[r]^-{T r} \ar[d]^{\pi_W} & T^H X \ar[d]^{\pi} \\
    U \ar@/^6pt/[ur]^{f_1} \ar@/_0pt/[ur]_(0.6){\!\!f_2}  \ar@/^0pt/[r]^{g_1} \ar@/_6pt/[r]_{g_2} 
                   & V \ar[r]^{h}            & W \ar[r]^{r}            & X .
  }
\]
Now $\pi_W \circ T h \circ f_1 = \pi_W \circ T h \circ f_2$,
so it makes sense to form $T h \circ f_1 + T h \circ f_2$.
This is smooth, and therefore so is
$p_1 + p_2 = T r \circ (T h \circ f_1 + T h \circ f_2)$.

$T^{dvs} X$ is defined to be the completion of $T^H X$ to a diffeological vector space over $X$.
It follows that $T^H X$ and $T^{dvs} X$ agree.
\end{proof}

While diffeological groups are weakly filtered, we will see in Example~\ref{ex:prodR} that they are not always filtered.
Nevertheless, in~\cite[Theorem~4.15]{CW} we used the group multiplication to show that for every diffeological group $G$, 
$T^H G = T^{dvs} G$. 
Now we extend this result to any homogeneous space:

\begin{prop}\label{prop:H=dvs-homogeneous}
For every homogeneous space $X$, we have $T^H X = T^{dvs} X$.
\end{prop}

Note that this proposition generalizes the last statement of Corollary~\ref{co:group-exact}.

\begin{proof}
Again, it is enough to show that addition $T^H X \times_X T^H X \ra T^H X$ is smooth. 
Assume that $X$ is diffeomorphic to $G/H$, where $G$ is a diffeological group and $H$ is a subgroup,
and let $\pi:G \ra X$ be the quotient map.
It is straightforward to check that we have a commutative square
\[
\xymatrix{T G \times_G T G     \ar[r] \ar[d] & T G \ar[d] \\
          T^H X \times_X T^H X \ar[r]        & T^H X,}
\]
where both horizontal functions are vector addition, and both vertical maps are induced by $\pi$.
By~\cite[Theorem~4.15]{CW}, the top function is smooth, so it suffices
to show that every plot $p:U \ra T^H X \times_X T^H X$ locally lifts to a plot 
$q:U \ra T G \times_G T G$.
Recall that the plot $p$ is equivalent to plots $p_1,p_2:U \ra T^H X$
such that 
\begin{equation}\label{eq:1}
\pi_X \circ p_1 = \pi_X \circ p_2,
\end{equation} 
where $\pi_X:T^H X \ra X$ is the natural projection.
Observe that any plot of $T^H X$ locally lifts to a plot of $T G$, since
the plots of $T^H X$ are generated by $T r$ for $r$ a plot of $X$, and the plots
of $X$ locally lift to $G$ since $X$ has the quotient diffeology.
Thus, for any $u \in U$, there exists an open neighbourhood $V$ of $u$ in $U$
such that $p_i|_V$ lifts to a plot $q_i : V \to TG$ for each $i$.
Consider the map $g:V \ra G$ sending $v$ to $(\pi_G (q_2(v)))^{-1} \cdot (\pi_G (q_1 (v)))$, 
where $\pi_G: TG \ra G$ is the projection.
By Equation~\eqref{eq:1}, $g(v)$ is in $H$ for $v \in V$.
Let $q'_2$ be the composite
\[
\xymatrix{V \ar[r]^-{(q_2,g)} & TG \times G \ar[r]^-{1 \times \sigma} & TG \times TG \ar[r] & TG,}
\]
where $\sigma:G \ra TG$ is the zero section, and the last map is induced by the multiplication of $G$.
It is easy to see that $\pi_G \circ q_1 = \pi_G \circ q'_2$, 
and that $\pi_* \circ q'_2 = p_2|_V$.
Therefore, $q:=(q_1,q'_2)$ is a lift of $p|_V$.
\end{proof}

Our next result makes use of the following concept:

\begin{de}\label{de:fine}
Let $V$ be a vector space.
The \dfn{fine} diffeology on $V$ is the smallest diffeology on $V$
making it into a diffeological vector space~(Definition~\ref{def:dvs}).
\end{de}

For example, the fine diffeology on $\R^n$ is the standard diffeology.

\begin{rem}\label{rem:fineness}
The fine diffeology is generated by the injective linear functions $\R^n \to V$ (\cite[3.8]{I}).
That is, the plots of the fine diffeology are the functions $p : U \to V$ such that 
for each $u \in U$, there is an open neighbourhood $W$ of $u$ in $U$,
an injective linear function $i : \R^n \to V$, for some $n \in \N$, and a smooth function $f : W \to \R^n$
such that $p|_W = i \circ f$.
\end{rem}

\begin{thm}\label{th:fine}
If $(X,x)$ is filtered, then $T_x(X)$ is a fine diffeological vector space.
\end{thm}

\begin{proof}
By the previous theorem, Hector's diffeology and the dvs diffeology agree,
and the dvs diffeology makes each $T_x(X)$ into a diffeological vector space,
so it is enough to show that Hector's diffeology is contained in the fine diffeology.
In other words, we need to show that any plot $p : U \to T_x(X)$ factors locally
as a smooth function to a finite-dimensional vector space $K$ (with the standard diffeology)
followed by a linear function $K \to T_x(X)$.

Our argument parallels the proof of Theorem~\ref{th:H=dvs}.
Every plot of $T_x(X)$ is locally of the form $Tq \circ f$, where $q : V \to X$ is
a plot and $f : U \to TV$ is smooth:
\[
  \xymatrix{
                   & TV \ar[r]^-{T q} \ar[d]^{\pi_V} & T^H X \ar[d]^{\pi} \\
    U \ar[ur]^{f} & V \ar[r]^{q}            & X .
  }
\]
By translating $U$ and $V$, we can assume that $0 \in U$, $0 \in V$, $\pi_V(f(0)) = 0$
and $q(0) = x$.
The plot $Tq \circ f$ must land in $T_x(X)$, so $q \circ \pi_V \circ f$ must be constant at $x$.
That is, $\pi_V \circ f$ can be viewed as a morphism in $\cG(X,x)$ from the constant
plot to $q$.  The $0$ germ is also a morphism between these pointed plots.
Since $X$ is filtered, by shrinking $U$ and $V$ if necessary, 
there exist a pointed plot $r : (W,0) \to (X,x)$ and a pointed map $h : (V,0) \to (W,0)$
such that $h \circ \pi_V \circ f$ is constant and $r \circ h = q$:
\[
  \xymatrix{
                   & TV \ar[r]^-{T h} \ar[d]^{\pi_V} & TW \ar[r]^-{T r} \ar[d]^{\pi_W} & T^H X \ar[d]^{\pi} \\
 U \ar[ur]^{f} & V \ar[r]^{h}            & W \ar[r]^{r}            & X .
  }
\]
Then $Th \circ f$ is a smooth function landing in the finite-dimensional vector space $T_0 W$,
providing the required local factorization.
\end{proof}

\begin{ex}\label{ex:prodR}
From~\cite[Example 4.24]{CW}, we know that $T_0(\prod_{\omega} \R)$ and $T_0(C^{\infty}(M, \R))$ are not fine,
where $M$ is a manifold of dimension at least 1.
It follows that the diffeological vector spaces $\prod_{\omega} \R$ and $C^{\infty}(M, \R)$ are not filtered.
\end{ex}

Our final result is that fine diffeological vector spaces are filtered, using the following
sufficient condition for a weakly filtered diffeological space to be filtered.
We say that a diffeological space $X$ is \dfn{injectively generated at $x$} if each 
pointed plot $(U,0) \to (X,x)$ factors as a germ at $0$ through an injective pointed plot $(V,0) \to (X,x)$.
We say that $X$ is \dfn{injectively generated} if it is injectively generated at $x$ for each $x \in X$,
i.e., if every plot locally factors through an injective plot.

\begin{lem}\label{lem:injective}
If $(X,x)$ is a weakly filtered pointed diffeological space and is injectively generated at $x$,
then $(X,x)$ is filtered.
\end{lem}

\begin{proof}
Given $f,g:r \ra p$ in $\cG(X,x)$, by assumption, 
there is $q \in G$, which is injective, together with a morphism $h:p \ra q$ in $\cG(X,x)$.
The injectivity of $q$ implies that $h \circ f = h \circ g$.
\end{proof}

\begin{cor}\label{co:fine-filtered}
Every fine diffeological vector space is filtered.
\end{cor}

\begin{rem}
Here is another application of Proposition~\ref{prop:dgrp}, Theorem~\ref{th:fine} and
Lemma~\ref{lem:injective}. If $G$ is a diffeological group such that $T_e(G)$ is not fine,
then $G$ is not injectively generated at $e$. In particular, neither $\prod_\omega \R$
nor $C^\infty(M,\R)$, with $M$ a manifold of dimension at least $1$, is injectively generated at $0$.
\end{rem}

\vspace*{10pt} %

\end{document}